\documentclass[12pt,a4paper]{amsproc}
\usepackage[english]{babel}
\usepackage{latexsym}
\usepackage{amsmath}
\usepackage{amssymb}
\usepackage{amscd}
\usepackage{xcolor}
\usepackage{marginnote}
\usepackage{soul}
\usepackage{hyperref}
\usepackage[cp850]{inputenc}
\usepackage[mathscr]{eucal}
\theoremstyle{plain}

\newtheorem{quest}{Question}
\newtheorem{question}[quest]{Question}
\newtheorem{prob}[quest]{Problem}
\newtheorem{defin}[quest]{Definition}
\newtheorem{theorem}[quest]{Theorem}
\newtheorem{prop}[quest]{Proposition}
\newtheorem{corollary}[quest]{Corollary}
\newtheorem{lemma}[quest]{Lemma}
\newtheorem{fact}[quest]{Fact}

\newtheorem{obs}[quest]{Observation}
\newtheorem*{prop*}{Proposition}
\newtheorem*{defin*}{Definition}
\theoremstyle{remark}
\newtheorem{example}[quest]{Example}

\newcommand{\R}{\mathbb{R}}
\newcommand{\N}{\mathbb{N}}

\newcommand{\C}{\mathbb{C}}
\newcommand{\T}{\mathbb{T}}

\newcommand{\adef}{\begin{defin}}
\newcommand{\zdef}{\end{defin}}
\newcommand{\F}{\mathcal F}
\newcommand{\TT}{\mathcal T}

\def\lspan{\operatorname{span}}
\def\Lip{\operatorname {Lip}}
\def\id{\operatorname {id}}

\newcommand{\aproof}{\begin{proof}}
\newcommand{\zproof}{\end{proof}}

\bigskip

\title{Equivariant liftings in Lipschitz-free spaces}

\thanks{V. Ferenczi's research has been supported by S\~ao Paulo Research Foundation (Fapesp), grant 2023/12916-1 and
by National Council for Scientific and Technological Development (CNPq),
grant 304194/2023-9. P. Kaufmann's research has been supported by S\~ao Paulo Research Foundation (Fapesp), grant 2023/12916-1. E. Perneck\'a's research was supported by the Czech Science Foundation (GA\v CR) grant 22-32829S}

\subjclass[2010]{Primary 46B20.   Secondary 46B04, 46B28.}

\keywords{Lipschitz maps, Lipschitz-free spaces, exact sequences of Banach spaces}

\author{Valentin Ferenczi, Pedro Kaufmann, and Eva Perneck\'a}

\address{Valentin Ferenczi, Departamento de Matem\'atica, Instituto de Matem\'atica e
Estat\'\i stica, Universidade de S\~ao Paulo, rua do Mat\~ao 1010,
05508-090 S\~ao Paulo SP, Brazil  \\ and \newline
Equipe d'Analyse Fonctionnelle \\
Institut de Math\'ematiques de Jussieu \\
Universit\'e Pierre et Marie Curie - Paris 6 \\
Case 247, 4 place Jussieu \\
75252 Paris Cedex 05 \\
France.}
\email{ferenczi@ime.usp.br}
\address{Pedro L. Kaufmann, Departamento de Ci\^encia e Tecnologia, Instituto de Ci\^encia e Tecnologia, Universidade Federal de S\~ao Paulo, Rua Monsueto Cesare Giulio Lattes, 1301, 12247-014, S\~ao Jos\'e dos Campos, SP, Brazil.}
\email{plkaufmann@unifesp.br}
\address{Eva Perneck\'a, Faculty of Information Technology, Czech Technical University in Prague, Th\'akurova 9, 160 00, Prague 6, Czech Republic}
\email{perneeva@fit.cvut.cz}

\begin{document}
\begin{abstract}
We consider Banach spaces $X$ that can be linearly lifted into their Lipschitz-free spaces $\mathcal{F}(X)$ and, for a group $G$ acting on $X$ by linear isometries, we study the possible existence of $G$-equivariant linear liftings. In particular, we prove that such lifting exists when $G$ is compact in the strong operator topology, or an increasing union of such groups and $\mathcal{F}(X)$ is complemented in its bidual by an equivariant projection. 
 As an example of application, we define and study a complex version of the Lipschitz-free space $\mathcal{F}(X)$ when $X$ is a subset of a complex Banach space stable under the action of the circle group.
\end{abstract}

\maketitle

\section{Introduction}

We begin by recalling the definitions of the spaces of Lipschitz functions and their natural preduals, the Lipschitz-free spaces. Given a metric space $(M,d)$ with a fixed base point $0\in M$, the vector space $\Lip_0(M)$ of all Lipschitz functions $f\colon M\to \R$ satisfying $f(0)=0$, endowed with the norm given by the least Lipschitz constant of a function 
$$\|f\|=\Lip(f)=\sup\left\{\frac{f(x)-f(y)}{d(x,y)},\,x, y\in M, x\neq y\right\},$$ is a Banach space. Moreover, it is isometric to a dual Banach space. Indeed, consider the isometric embedding $\delta\colon M\to {\rm Lip}_0(M)^*$ defined by $\langle f,\delta(x)\rangle=f(x)$ for all $x\in M$ and $f\in \Lip_0(M)$. The norm-closed subspace of $\Lip_0(M)^*$ generated by $\delta(M)$,
$$\F(M)=\overline{\lspan\,\delta(M)}^{\|\cdot\|}\,\subset\,\Lip_0(M)^*,$$
is called the {\emph{Lipschitz-free space}} over $M$. It can be characterized by its universal property as the unique, up to a linear isometry, Banach space that contains an isometric copy of $M$ and such that any Lipschitz map $L\colon M\to X$ into a Banach space $X$ with $L(0)=0$ can be uniquely extended to a bounded linear operator $\bar{L}\colon \F(M)\to X$, meaning that $\bar{L}  \delta=L$ (see \cite[Theorem 3.6]{weaver}). Since $\|\bar{L}\|=\Lip(L)$, by extending real-valued Lipschitz functions on $M$ we obtain that $\F(M)^*$ is linearly isometric to $\Lip_0(M)$. The universal property also provides any Lipschitz map $L\colon M\to N$ between two metric spaces $M, N$ that preserves the base point with a linearization to a unique bounded linear operator $\widetilde{L}\colon \F(M)\to\F(N)$ between the corresponding Lipschitz-free spaces such that $\|\widetilde{L}||=\Lip(L)$; simply let $\widetilde{L}=\overline{\delta  L}$. Moreover, Lipschitz homeomorphisms (resp. isometries) between metric spaces induce linear isomorphisms (resp. isometries) between the Lipschitz-free spaces. When the metric space is a Banach space $X$, there is additional structure relating $X$  and $\F(X)$. Namely, we can define the barycenter operator
$$\beta_X:  \F(X) \rightarrow X$$
as the linear extension of the identity map on $X$, i.e. $\beta_X=\overline{\id_X}$ (see \cite{godefroy2003lipschitz} after Lemma 2.3). In some cases (see next section), this map admits a linear right inverse, providing a linear copy of $X$ inside $\F(X)$. When the setting is clear, we will sometimes omit the subscript $X$ from the notation of $\beta$. For more on the theory of spaces $\Lip_0(M)$ and $\F(M)$, we refer to the book of Weaver \cite{weaver} (where the latter are called {\emph{Arens-Eells spaces}  and denoted $\AE(M)$) or to Godefroy and Kalton's paper \cite{godefroy2003lipschitz}.

The main aim of this paper, given a group $G$ acting by linear isometries (or simply acting boundedly by automorphisms) on a Banach space $X$,
is to investigate the possible $G$-equivariance of a bounded linear right inverse of the map $\beta_X: \F(X) \rightarrow X$, when such an inverse exists (i.e. when $X$ has the lifting property defined in \cite{godefroy2003lipschitz}).

The question may initially be described in the more general setting of pointed metric spaces. If $M$ is a metric space with a base point $0$ and $G \times M \rightarrow M$ is an action of a group $G$ on $M$ by Lipschitz homeomorphisms (resp. isometries) fixing $0$, the map $\delta: M\to \F(M)$ is what we call {\em $G$-equivariant}, i.e. $\tilde{g}\delta=\delta g$ holds for every $g\in G$ and the induced isomorphism (resp. isometry) $\tilde{g}$. This means informally that $\F(M)$ inherits the metric "symmetries" of $M$.   More explicitely, if in the universal property characterizing $\F(M)$, the Lipschitz map $L: M \to X$ is $G$-equivariant, then the arrows $\delta, \overline{L}$  in the  diagram corresponding to the relation $\bar{L}  \delta=L$        are $G$-equivariant as well. For $\overline{L}$ this means that $\overline{L}\tilde{g}=g\overline{L}$ and is verified easily in every $\delta(x)$ and then in an arbitrary element of $\F(X)$ by linearity and continuity. Therefore, it is possible to check the following: {\em $\F(M)$ is also a universal object in the category where objects are required to be equipped with a bounded action of $G$ and arrows are required to be  $G$-equivariant}.

In the case when the metric space is a Banach space $X$, the map $\beta_X$ is also $G$-equivariant in the sense that $\beta_X \tilde{g}=g\beta_X$ for all $g \in G$. However, linear right inverses of $\beta_X$ are not, in general, $G$-equivariant. Our aim is therefore to study whether the symmetries of $X$ appearing in some well-chosen linearly isomorphic copy of $X$ inside $\F(X)$, when such a copy exists,  can be "compatible" with those same symmetries induced on $\F(X)$ as described above. In other words, we study the compatibility between the linear theory of the Lipschitz-free space $\F(X)$ and the action of a group on the Banach space $X$.

For recent work on group actions and Lipschitz-free spaces or spaces of Lipschitz functions on metric spaces, see \cite{cuthdoucha, cuthdouchatitkos, raunig, rosendal}.

\section{The $G$-equivariant lifting property}

Following the notation of \cite{CF}, we call {\em $G$-space} a Banach space $X$ together with an action $(g,x) \mapsto gx$ of a group $G$ by automorphisms on $X$. A map $T$ between $G$-spaces $X$ and $Y$ is $G$-equivariant if $T(gx)=gTx$ for all $x \in X$.

The following definition and theorem are from \cite{godefroy2003lipschitz}.

\begin{defin}\label{lifting} A Banach space $X$ has the (isometric)  lifting property if there exists a (norm $1$) linear bounded right inverse to $\beta_X$, i.e. a map
$T: X \rightarrow \F(X)$ such that 
$\beta_X T=Id_X$. 
\end{defin}

\begin{theorem}[Godefroy-Kalton 2003] If $X$ is separable, then $X$ has the isometric lifting property.
\end{theorem}

We shall be interested in $G$-equivariant maps $T$
from $X$ to ${\mathcal F}(X)$, i.e. such that $$T g=\tilde{g} T$$ for all $g \in G$. The natural extension of the lifting property to the $G$-space setting is given by the next definition.

\begin{defin}
We say that a $G$-space $X$ has the \emph{$G$-equivariant (isometric)  lifting property} if there exists a (norm $1$) $G$-equivariant linear bounded right inverse for $\beta_X$, i.e. a $G$-equivariant map $T: X \rightarrow \F(X)$ such that 
$\beta_X T=Id_X$.
\end{defin}

\begin{obs} If $X$ has the $G$-equivariant lifting property via $T$, then the projection $P$ from $\F(X)$ onto $TX$ defined as
$$P=T\beta_X$$ commutes with $\tilde{g}$ for all $g \in G$.
In particular, $Im P$ is (globally) invariant under $\tilde{G}:=\{\tilde{g}, g \in G\}$. \end{obs}

It is clear that not every linear right inverse for $\beta$ is $G$-equivariant in general, but that in many cases one can be chosen to have this property. 
To see this, consider the example of $X=\R$ and $G=\{-Id,Id\}$, with the lifting map 
$T(\lambda)=\lambda \delta(1)$.
While this map is a linear inverse of $\beta_\R$, it is clear that it is not $G$-equivariant: indeed, $T(-1)=-\delta(1)$ while $\widetilde{(-Id)} T(1)=\widetilde{(-Id)}\delta(1)=\delta({-1})$. It is also the case that $\widetilde{(-Id)} T(1)$ does not even belong to $T(\R)=\R \delta(1)$.
However a $G$-equivariant linear inverse to $\beta_\R$ may be easily defined: take $U(\lambda)=\frac{\lambda}{2}(\delta(1) -\delta({-1}))$.

The main question of this paper is therefore as follows:

\begin{quest} If $X$ admits a bounded (resp. isometric) linear action of a group $G$ and has the (resp. isometric) lifting property, then which conditions on $X$ and/or $G$ guarantee that the $G$-equivariant (resp. isometric) lifting property holds?
\end{quest}

We can also ask whether it is possible to achieve the weaker condition that $\tilde{g}Tx$ belongs to $TX$ for all $g \in G$, $x \in X$.

\subsection{Homological interpretation and exact sequences of $G$-spaces}

We refer to the book by Cabello and Castillo
\cite{CC} for homological methods in Banach spaces, whose language is widely used in \cite{godefroy2003lipschitz}.
If $X$ is a Banach space and $\F(X)$ its Lipschitz-free space, the definition of the map $\beta_X$ corresponds to the exact sequence:

$$0 \rightarrow {\rm Ker} \beta_X \rightarrow {\mathcal F}(X) \rightarrow X \rightarrow 0.$$
In this context, the space $X$ admitting the lifting property of Godefroy-Kalton (Definition \ref{lifting}) means that the quotient map  $\beta_X: {\mathcal F}(X) \rightarrow X$ in this exact sequence admits a bounded linear inverse $T$, or equivalently that ${\rm Ker} \beta$ is complemented, for which in homological language we say that the exact sequence "splits".

\

Let $G$ be a group. In \cite{CF}, the related category of $G$-spaces, with $G$-equivariant arrows, and their associated short exact sequences, were studied, see \cite{CF} Section 10. The  comments in our introduction show that
$$0 \rightarrow {\rm Ker} \beta_{X}\rightarrow {\mathcal F}(X) \rightarrow X \rightarrow 0$$
is an exact sequence of $G$-spaces
(because the action of $G$ on ${\mathcal F}(X)$ leaves ${\rm Ker} \beta_{X}$ invariant). The question of whether there is a $G$-equivariant linear bounded right inverse of $\beta_{X}$ corresponds to the general notion of $G$-splitting of an exact sequence of $G$-spaces, \cite{CF} Definition 10.2 and Proposition 10.6, which may be formulated as follows:

\begin{defin}[Castillo-Ferenczi] Consider an exact sequence
$$0 \rightarrow Y \rightarrow Z \rightarrow X \rightarrow 0$$
where the map $i: Y \rightarrow Z$ is the inclusion map of the subspace $Y$ inside $Z$. Assume a group $G$ acts boundedly on $Z$ and leaves $Y$ invariant, and therefore that the quotient map $\pi: Z \rightarrow X$ is $G$-equivariant. Then the following assertions are equivalent:
\begin{itemize}
\item[(i)] The quotient map $\pi$ admits a $G$-equivariant linear bounded inverse $T$.
\item[(ii)] $Y$ admits a $G$-invariant complement $W$.
\item[(iii)] $Y$ is complemented in $Z$ by a $G$-equivariant projection $p$.
\end{itemize}
When (i)-(ii)-(iii) hold we say that the exact sequence $G$-splits, or splits as a sequence of $G$-spaces.
\end{defin}

It may be interesting to recall the relation between (i),(ii), and (iii): (i) $\Rightarrow$ (ii) Pick $W=TX$. (ii) $\Rightarrow$ (iii) Let
$p$ be the projection onto $Y$ associated to the decomposition $Z=Y \oplus W$.
$(iii) \Rightarrow (i)$ Let $Tx=(Id-p)\tilde{x}$ where $\tilde{x}$ is any element of $Z$ such that $\pi(\tilde{x})=x$.

In conclusion:

\begin{fact}
A $G$-space $X$ has the $G$-equivariant lifting property if and only if the exact sequence of $G$-spaces
$$0 \rightarrow {\rm Ker} \beta_{X}\rightarrow {\mathcal F}(X) \rightarrow X \rightarrow 0$$
 $G$-splits.
\end{fact}

\subsection{Duality}

The  notion of dual sequence
$$0 \rightarrow X^* \rightarrow Z^* \rightarrow Y^* \rightarrow 0$$
of an exact sequence
$$0 \rightarrow Y \rightarrow Z \rightarrow X \rightarrow 0$$
 is classical; the arrows of the dual sequence are simply the adjoints of the arrows of the original sequence, and it is an exercise to check that this defines an exact sequence.

In the context of Lipschitz-free spaces, the dual sequence to the exact sequence
$0 \rightarrow {\rm Ker} \beta_X \rightarrow \F(X) \rightarrow X \rightarrow 0$ is
$$0 \rightarrow X^* \rightarrow {\rm Lip}_0(X) \rightarrow {\rm Lip}_0(X)/X^* \rightarrow 0.$$

It is well-known that for any Banach space $X$, there exists a norm one projection from ${\rm Lip}_0(X)$ onto $X^*$, see e.g. \cite{benyamini1998geometric}
Proposition 7.5. In other words, 
\begin{fact} The sequence $0 \rightarrow X^* \rightarrow {\rm Lip}_0(X) \rightarrow {\rm Lip}_0(X)/X^* \rightarrow 0$ always splits in the category of Banach spaces. \end{fact}

This can also be seen as a particular case of \cite{godefroy2003lipschitz} Proposition 2.6.

Passing to the category of $G$-spaces, an exact sequence of $G$-spaces has dual exact sequence which is also of $G$-spaces; see for example \cite{garcia} Exemplo 3.1.9.
Namely, if 
$$0 \rightarrow Y \rightarrow X \rightarrow Z \rightarrow 0$$
is an exact sequence of $G$-spaces, with respective actions $g \mapsto u(g), \lambda(g), v(g)$ on $Y, X, Z$, then the dual sequence
$$0 \rightarrow Z^* \rightarrow X^* \rightarrow Y^* \rightarrow 0$$
is an exact sequence of $G$-spaces under the respective actions
$g \mapsto v(g^{-1})^*, \lambda(g^{-1})^*, u(g^{-1})^*$.
Note that if $T$ is a $G$-equivariant projection from $X$ onto $Y$, then  $T^*$ is a $G$-equivariant lifting from $Y^*$ into $X^*$. Therefore:

\begin{fact}
If an exact sequence of $G$-spaces $G$-splits, then the dual sequence also $G$-splits.
\end{fact}

These considerations lead us to the following definition and observation:

\begin{defin} A $G$-space $X$ has the $G$-equivariant dual lifting property if
there is a bounded and $G$-equivariant
projection from $\Lip_0(X)$ onto $X^*$; Equivalently, if
the sequence
$$0 \rightarrow X^* \rightarrow \Lip_0(X) \rightarrow \Lip_0(X)/X^* \rightarrow 0$$
of $G$-spaces $G$-splits.
\end{defin}

\begin{obs} If a $G$-space has the $G$-equivariant lifting property, then it has the dual $G$-equivariant lifting property.
\end{obs}

The converse is false in general: take $G=\{Id\}$. Any space has the dual $\{Id\}$-equivariant lifting property, but any space failing the lifting property obviously fails the $\{Id\}$-equivariant lifting property. See \cite{godefroy2003lipschitz} Theorem 4.3 for examples of such spaces.

\begin{question}
Find a $G$-space which fails the dual $G$-equivariant lifting property. \end{question}

\subsection{Examples}

\begin{example}
\label{ex:unconditional_basis}
If $X$ has a $1$-unconditional basis, the group $U \simeq \{-1,1\}^\omega$ of multiplications by signs on each coordinate will act also as isometries on $\F(X)$. We shall address the question of the existence of a $U$-equivariant lifting in the next section, where we provide two proofs of an affirmative answer.
\end{example}

\begin{example}
If $M$ is a metric space with a base point $0$ and $G$ is a group acting on $M$ by isometries fixing $0$, then $\F(M)$ has the $\tilde{G}$-equivariant isometric lifting property, where $\tilde{G}$ is the group of linear extensions of isometries from $G$ to $\F(M)$, $\tilde{G}=\{\tilde{g}, g \in G\}$. 
\end{example}
\begin{proof}
By \cite[Lemma 2.10]{godefroy2003lipschitz}, the Lipschitz-free space $\F(M)$ has the isometric lifting property with a lifting given by the linear extension $T$ of the isometric embedding $\delta_{\F(M)}\delta_M: M\to \F(\F(M))$, i.e. $T:\F(M)\to\F(\F(M))$ is the linear mapping satisfying $T\delta_M(p)=\delta_{\F(M)}\delta_M(p)$ for every $p\in M$. (The statement in \cite{godefroy2003lipschitz} is given only for $M$ being a Banach space, but the argument works in general). We will verify that $T$ is $\tilde{G}$-equivariant. Indeed, let $g\in G$ and $p\in M$. Then
$$
T\tilde{g}(\delta_M(p))=T\delta_M(g(p))=\delta_{\F(M)}\delta_M(g(p))=\delta_{\F(M)}\tilde{g}\delta_M(p)
=\tilde{\tilde{g}}\delta_{\F(M)}\delta_M(p)= \tilde{\tilde{g}}T(\delta_M(p))   
$$
Because $\overline{\lspan\delta_M(M)}=\F(M)$, by linearity and continuity $T\tilde{g}=\tilde{\tilde{g}}T$ on $\F(M)$.
\end{proof}

\section{First results: actions of compact groups}
In what follows we fix a $G$-space $X$ (and its associated Lipschitz-free space $\F(X)$).
\begin{defin} For $C \geq 1$ we let $\TT$ (resp. $\TT_C$) be the class of linear right inverses to $\beta_X$ (resp. of norm at most $C$).

We let $\TT(G) \subseteq \TT$ (resp. $\TT_C(G) \subseteq \TT_C$) be the class of $G$-equivariant linear right inverses to $\beta_X$ (resp. of norm at most $C$). \end{defin}

The following observation is immediate:

\begin{obs} If $G$ is finite and $X$ has the lifting property
then $\TT(G) \neq \emptyset$. \end{obs}
  \begin{proof} Just fix some $T_0$ in $\TT$, let
$$T={\rm Ave}_{g \in G} \tilde{g} T_0 g^{-1},$$
and note that for any $g,h\in G$ we have  $\widetilde{h}\widetilde{g}=\widetilde{hg}$ (this is easy to verify for every $\delta(x)$ and then follows by linearity and continuity).
\end{proof}

\subsection{The compact case}
The question now is whether we can define this kind of average in a meaningful way when $G$ is infinite. As we shall now investigate, this is possible in particular when we impose compactness with respect to the strong operator topology (SOT) induced by the space of automorphisms of $X$, which is denoted by $B(X)$. We begin by a simple continuity statement.

\begin{lemma}\label{SOTSOT}
Given a Banach space $X$, the map $g \mapsto \tilde{g}$ is continuous from $(B(X),SOT)$ to $(B(\F(X)),SOT)$.
\end{lemma}

\begin{proof}
Let $g_{\alpha}\xrightarrow{SOT} g$ in $G$. Fix any $x \in X$. Then
$$\|\widetilde{g_{\alpha}}(\delta(x))-\widetilde{g}(\delta(x))\|=\|\delta(g_{\alpha}(x))-\delta(g(x))\|=\|g_{\alpha}(x)-g(x)\|\to 0.$$
So, from the uniform boundedness and linearity of $g_{\alpha}$ and $g$ we get that 
$\widetilde{g_{\alpha}}\xrightarrow{SOT} \widetilde{g}$ as well. (Note that in fact the converse also holds.)
\end{proof}

Next we will extend the averaging argument to the case when $G$ is compact in SOT.
\begin{prop}\label{SOTcompact}
Let $X$ be a Banach space with the lifting property and $G$ be a bounded, SOT compact group of automorphisms on $X$ with associated Haar measure $\mu$. Then for any $T_0\in\TT$, the mapping $T:X\to\F(X)$ defined by
$$T(x)=\int_{g \in G} \tilde{g} T_0 g^{-1}(x) d\mu\quad\textup{for every}\,\,x\in X$$ is a $G$-equivariant element of $\TT$.
Furthermore, if $T_0 \in \TT_1$ and $G$ is a group of isometries, then $T \in \TT_1$.
\end{prop}
\begin{proof}
First, we show that for every $x\in X$, $T(x)$ is a well-defined element of $\F(X)$. Let $g_{\alpha}\xrightarrow{SOT} g$ in $G$ and let $\sup\{\|h\|, h\in G\}\leq C<\infty$. Then for any $x\in X$ we have 
$$\|g_{\alpha}^{-1}(x)-g^{-1}(x)\|=\|g_{\alpha}^{-1}gg^{-1}(x)-g_{\alpha}^{-1}g_{\alpha}g^{-1}(x)\|\leq C\|gg^{-1}(x)-g_{\alpha}g^{-1}(x)\|\to 0,$$
thus also $g_{\alpha}^{-1}\xrightarrow{SOT} g^{-1}$. By Lemma \ref{SOTSOT} we get that 
$\widetilde{g_{\alpha}}\xrightarrow{SOT} \widetilde{g}$ as well. Therefore, for a fixed $x\in X$, the function $g\mapsto \tilde{g} T_0 g^{-1}(x)$ from $G$ to $\F(X)$ is SOT to norm continuous, and hence $\mu$-measurable \cite[Proposition 2.15]{Ryan_2002}. Indeed,
\begin{align*}
\|\widetilde{g_{\alpha}} T_0 g_{\alpha}^{-1}(x)-\widetilde{g} T_0 g^{-1}(x)\|&=\|\widetilde{g_{\alpha}} T_0 g_{\alpha}^{-1}(x)-\widetilde{g_{\alpha}} T_0 g^{-1}(x)+\widetilde{g_{\alpha}} T_0 g^{-1}(x)-\widetilde{g} T_0 g^{-1}(x)\|\\
&\leq\|\widetilde{g_{\alpha}}\|\|T_0\|\|g_{\alpha}^{-1}(x)-g^{-1}(x)\|+\|\widetilde{g_{\alpha}} T_0 g^{-1}(x)-\widetilde{g} T_0 g^{-1}(x)\|\\
&\leq C\|T_0\|\|g_{\alpha}^{-1}(x)-g^{-1}(x)\|+\|\widetilde{g_{\alpha}} (T_0 g^{-1}(x))-\widetilde{g} (T_0 g^{-1}(x))\|\to 0,
\end{align*}
whenever $g_{\alpha}\xrightarrow{SOT}g$. Moreover, since 
\begin{equation}
\label{eq:norm_of_averaged_operator}
\|\tilde{g} T_0 g^{-1}(x)\|\leq C^2\|T_0\|\|x\|,
\end{equation}
the real-valued function $g\mapsto \|\tilde{g} T_0 g^{-1}(x)\|$ is SOT to $|\cdot|$ continuous and bounded. Hence it is $\mu$-integrable, and by the Bochner's Theorem \cite[Proposition 2.16]{Ryan_2002} we conclude that also $g\mapsto \tilde{g} T_0 g^{-1}(x)$ is Bochner $\mu$-integrable. This verifies that $T(x)\in\F(X)$. 

The linearity of the operator $T:X\to\F(X)$ is straightforward and \eqref{eq:norm_of_averaged_operator} implies that $\|T\|\leq C^2\|T_0\|$. To observe that $T$ is actually a lifting of $\beta$, we apply Proposition 2.18 from \cite{Ryan_2002} saying that the operators respect the Bochner integral:
$$\beta T(x)=\int_{g \in G} \beta\widetilde{g} T_0 g^{-1}(x) d\mu=\int_{g \in G} g\beta T_0 g^{-1}(x) d\mu=\int_{g \in G} g g^{-1}(x) d\mu=\int_{g \in G} x d\mu=x,$$
where the second equality follows from the $G$-equivariance of $\beta$ and the third one from the assumption that $T_0\in\TT$. We have thus obtained that $T\in\TT$. 

Finally, let $h\in G$. Recall that $\widetilde{h}\widetilde{g}=\widetilde{hg}$ for any $g\in G$. By \cite[Proposition 2.18]{Ryan_2002} again, we may write
\begin{align*}
\widetilde{h}T(x)&=\int_{g \in G} \widetilde{h}\tilde{g} T_0 g^{-1}(x) d\mu=\int_{g \in G} \widetilde{hg} T_0 g^{-1}(x) d\mu\\
&=\int_{f \in G} \widetilde{f} T_0 f^{-1}h(x) d\mu=Th(x),
\end{align*}
where we substituted $f=hg$ and relied on the left-invariance of $\mu$. This shows that $T$ is $G$-equivariant and finishes the proof of the main statement. The furthermore assertion follows immediately from \eqref{eq:norm_of_averaged_operator}. 
\end{proof}

\subsection{A complex version of the Lipschitz-free space when $M$ admits an action of the unit circle} 
We denote by $\T$ the unit circle as a multiplicative group. If $M$ is a pointed metric space with an isometric action of $\T$ fixing $0$,
$X$ is any complex Banach space, and $L$ is a Lipschitz map from $M$ into $X$ sending $0$ to $0$ and which is $\T$-equivariant, i.e. $L(e^{i\theta}.m)=e^{i\theta} Lm$ for all $m \in M$ and $\theta \in \R$, then 
we have the factorization $L=\overline{L} \delta$ with $\T$-equivariant maps. 
Note however that the induced action of $\T$ on $\F(M)$ does not turn it into a complex Banach space (in particular, we have $\tilde {i}.(\tilde{i}.\delta(m))=\delta((-1).m)$ instead of $-\delta(m)$).
Therefore, we consider the subspace 
$$Y_\T=\overline{\rm span}\{\delta(e^{i\theta}.m)-\cos\theta \delta(m) - \sin\theta \delta(i.m), m \in M, \theta \in \R\}$$
and the quotient $\F_{\C}(M):=\F(M)/Y_\T$. This is a complex space, under the complex law defined by $ i.(\delta(m)+Y_\T)=\delta(i.m)+Y_\T$
(in particular note that $e^{i\theta}$ acts isometrically on $\F_{\C}(M)$). If $\pi_\C: \F(M) \rightarrow \F_\C(M)$ is the associated quotient map, then 
the map $\delta_{\C}:=\pi_\C \delta$ is $1$-Lipschitz and $\T$-equivariant  
from $M$ into $\F_\C(M)$. Noting that $Y_\T \subseteq {\rm Ker}(\overline{L})$, we have the factorization $L=\tilde{L}\delta_\C$, where 
$\tilde{L}(\delta(m)+Y_\T):=\overline{L}(\delta(m))=Lm$.
We have $\|\delta_\C(m)-\delta_\C(n)\|=\sup_{f \in Y_\T^{\perp}, \|f\|=1} |f(m)-f(n)|$, and therefore $\delta_\C$ has in general no reason to be isometric or even injective.
However:

\begin{obs}\label{isometryT} If $M$ is a subset of a complex Banach space $X$ and the action of $\T$ on $M$ is induced  by the complex law in $X$, then $\delta_\C: M \rightarrow {\mathcal F}_{\C}(M)$ is isometric. Moreover, for any complex Banach space $Y$ and any $\T$-equivariant Lipschitz map $L: M\to Y$ such that $L(0)=0$, we have a factorization $L=\tilde{L}\delta_\C$ with a bounded linear operator $\tilde{L}:\F_\C(M)\to Y$ satisfying $\|\tilde{L}\|=\Lip(L)$.

\end{obs}
\begin{proof} 
Let $\beta_M: \F(M) \rightarrow X$ be the restriction of $\beta_X$ to $\F(M)$, and note that $Y_\T \subseteq ker \beta_M$; so we also have a  $\C$-linear map
 $\beta_{\C,M}: \F_\C(M) \rightarrow X$ defined as a quotient map by $\beta_{\C,M}(\phi+Y_\T) = \beta_M(\phi)$, and which has norm $1$. The relation 
 $\beta_{\C,M}\delta_{\C}=Id_M$ holds, and $\delta_\C$ is therefore an isometric map in this case. By a direct computation, we can verify that the linear extension of the map $\tilde{L}(\delta(m)+Y_\T):=\overline{L}(\delta(m))=Lm$ has norm equal to $\Lip(L)$.
\end{proof}

 Alternatively, consider the $\sqrt{2}$-Lipschitz map $L: M \rightarrow \F(M)$ defined by
 $$L(m):=\int_{\theta}(\cos\theta\delta(e^{-i\theta}m)+\sin\theta\delta(ie^{-i\theta}m))
 \frac{d\theta}{2\pi}.$$ This is well defined because the mapping $f_m:[0,2\pi]\to\F(M)$ given by $f_m(\theta)=\cos\theta\delta(e^{-i\theta}m)+\sin\theta\delta(ie^{-i\theta}m)$ is continuous and with the range in a separable (sub)space, hence measurable, and bounded, hence Bochner integrable.
 It can  be checked  
 that
 $$L(e^{i\alpha}.m)=\cos\alpha L(m)+\sin\alpha L(i.m), \forall \alpha \in \R.$$

 Indeed the change of variable $\theta \mapsto \theta+\alpha$ easily gives
 $$L(e^{i\alpha}.m)-\cos\alpha L(m)=\sin\alpha
 \big(\int_\theta \cos(\theta)\delta(e^{-i\theta}im)-\sin(\theta)\delta(e^{-i\theta}m)\frac{d\theta}{2\pi}\big)$$
 Regarding the second integral on the right, the change of variable $\theta \rightarrow \theta+\pi$ gives
 $$\int_\theta -\sin(\theta)\delta(e^{-i\theta}m)\frac{d\theta}{2\pi}=\int_\theta \sin(\theta)\delta(-e^{-i\theta}m)\frac{d\theta}{2\pi}$$
 and therefore, as expected,
 $L(e^{i\alpha}.m)-\cos\alpha L(m)=\sin\alpha L(i.m).$

 Define a linear map $Q$ on $\F(M)$ by $Q=\overline{L}$, i.e. $$Q\delta(m):=\int_{\theta}(\cos\theta\delta(e^{-i\theta}m)+\sin\theta\delta(ie^{-i\theta}m))
 \frac{d\theta}{2\pi}.$$ This map has norm at most $\sqrt{2}$, and we have
 $Q\delta(e^{i\theta}.m)=\cos\theta Q\delta(m)+\sin\theta Q\delta(i.m)$. 
 In other words, $Y_\T \subseteq Ker Q$.
 Let $P=Id-Q$. We note the formula 
 $$P\delta(m)=\int_\theta (\delta(m)-
 \cos\theta\delta(e^{-i\theta}m)-\sin\theta\delta(ie^{-i\theta}m))
 \frac{d\theta}{2\pi} \in Y_\T,$$ and it follows that $Im P \subseteq Y_\T$. This implies that $QP=0$, i.e. $P$ (and $Q$) are projections. Therefore $Im P = Ker Q$ and therefore $Y_\T=Ker Q$ as well.
 Therefore we have an isomorphic identification of $Z_\T:={\rm Im}Q$ with $\F_\C(M)$ given by $z \mapsto z+Y_\T$. Summing up:

\begin{prop} The space $\F_\C(M)$ is linearly isomorphic to the complemented subspace $Z_\T$ of $\F(M)$ defined as
$$Z_\T=\overline{{\rm span}}\left\{\int_{\theta}(\cos\theta\delta(e^{-i\theta}m)+\sin\theta\delta(ie^{-i\theta}m))
 \frac{d\theta}{2\pi},\, m \in M\right\}.$$

\end{prop}

For a quite different approach to a complex version of $\F(M)$, involving Lipschitz maps with complex values instead of real, see \cite{ACP}. 

 Recall that Godefroy and Kalton showed that all separable real Banach spaces have the isometric lifting property (\cite[Theorem 3.1]{godefroy2003lipschitz}). We derive a complex version of that result: 
 
 \begin{prop}
 If $X$ is a separable complex space, then $X$ is $\C$-linearly isometric to a $1$-complemented subspace of the complex space $\F_\C(X)$.
 \end{prop}
 
 \begin{proof}
 Since $\T$ is compact, Proposition \ref{SOTcompact} applies: there exists an $\R$-linear lifting $T: X \rightarrow \F(X)$ which is $\T$-equivariant and isometric. The induced map $T_\C=\pi_\C T$ is $\C$-linear from $X$ into $\F_\C(X)$.
 Since $\beta_{\C,X} T_\C=Id_X$, this map is also an isometric embedding, and $X$ is $\C$-linearly isometric to a $1$-complemented subspace of $\F_\C(X)$
 (with projection $T_\C \beta_{\C,X}$).
 \end{proof}

A notable consequence of the aforementioned result \cite[Theorem 3.1]{godefroy2003lipschitz} is that, whenever $X$ is a real Banach space admitting an isometric embedding into a real Banach space $Y$, then $Y$ contains a linear subspace which is isometric to $X$ (\cite[Corollary 3.3]{godefroy2003lipschitz}). 
This does not hold in the complex case: Kalton's space $Z_2(\alpha), \alpha \neq 0,$ \cite{KaltonZ2alpha} does not $\C$-linearly embed (not even isomorphically) into its conjugate  \cite{cuellar}, although it is obviously ($\R$-linearly) isometric to it. See also \cite{ferenczi} for a complex space which is even totally incomparable with its conjugate. However the following seems to remain open: 

\begin{question}
Suppose that $X$ is a separable complex Banach space admitting a {\em $\mathbb C$-homogeneous} isometric embedding into a complex Banach space $Y$. Does $Y$ contain a subspace $\mathbb C$-isometrically isomorphic to $X$?  
\end{question}

Since the result of \cite{godefroy2003lipschitz} in the real case depends on arguments by Figiel \cite{figiel} which work for real Banach spaces, a different approach would be needed to investigate the question above.

\begin{question} Assume $\{0\} \cup \T$ is equipped with the natural action of $\T$.
Is the space $\F_\C(\{0\} \cup \T)$ isomorphic to $L_1([0,1],\C)$?
\end{question}

The space $\F(\{0\} \cup \T)$ is linearly isomorphic to $L_1=L_1([0,1],\R)$
(to see this, combine \cite[Lemma 2.8 (iii) and Theorem 4.21]{AACD21}} with \cite{Godard}).
So what we know about the above question is that the associated space $\F_\C(\{0\} \cup \T)$
is $\R$-linearly isomorphic to a complemented subspace of $L_1$.

\

The case of Hilbertian spaces is also particularly interesting. Indeed, the transitivity of the action of the unitary group $U(H)$ on the sphere of a Hilbert space implies that given $x \in H$ and $\phi \in \F(H)$  such that $\beta_{H}(\phi)= x$, there can be at most one $U(H)$-equivariant lifting $T$ such that $T(x)=\phi$. Furthermore:

\begin{obs} Assume $T: H \rightarrow \F(H)$ is a $U(H)$-equivariant lifting of $\beta_H$, and $x \in H$. Let $\phi=Tx \in \F(H)$. Let $u$ be any unitary operator on $[x]^{\perp}$, and $\tilde{u} \in B(H)$ be defined by the matrix $=\begin{pmatrix} Id & 0 \\ 0 & u \end{pmatrix}$, relatively to the decomposition
$H=[x] \oplus [x]^{\perp}$. Then
$\phi$ is invariant under $\tilde{u}$, i.e. for any $f \in Lip_0(H)$, $\phi(f)=\phi(f   \tilde{u})$.
\end{obs}

\begin{prob} Let $H$ be a Euclidean space. 
\begin{itemize}
\item 
Describe the set of $\phi$
 in $\F(H)$ such that $\|\beta_H (\phi)\|=\|\phi\|$.
\item Describe the subset $S$
of such $\phi$ for which there is a (necessarily unique) $U(H)$-equivariant isometric lifting $T$ such that $T \beta_H \phi=\phi$.
\item Describe the set $S' \supseteq S$ of $\phi \in \F(H)$ for which there is a (necessarily unique) $U(H)$-equivariant lifting $T$ such that $T \beta_{H}(\phi)=\phi$.
\end{itemize}\end{prob}

 Example \ref{ex:unconditional_basis} is a consequence of  Proposition \ref{SOTcompact} since the group of units on a space with a $1$-unconditional basis is SOT-compact.  However, we give a direct proof now, which may have generalizations and is inspired by the methods of \cite{godefroy2003lipschitz}.

\begin{prop} Let $X$ be an infinite dimensional Banach space with $1$-unconditional basis, and let $U$ be the group of units $\{-1,1\}^\omega$ acting in the canonical way on $X$. Then $X$ has the $U$-equivariant lifting property. \end{prop}

\begin{proof}

 The first part of the proof is an adaptation of \cite[Theorem 3.1]{godefroy2003lipschitz}. Let $\{e_n\}$ be an unconditional basis for $X$, and let $\alpha_n>0$ be such that $\{\sum_{n=1}^\infty t_n\alpha_ne_n, t_n\in [-1/2,1/2] \}$ is compact. Consider on  $H=[-1/2,1/2]^\mathbb{N}$ the product topology and measure $\lambda$ which is the product of Lebesgue measures in each factor, and define, for each $t\in H$, $L(t)= \sum_{n=1}^\infty t_n\alpha_n e_n$. Let $\mathbb{N}_n=\mathbb{N}\setminus\{n\}$, $H_n=[-1/2,1/2]^{\mathbb{N}_n}$ and again $\lambda_n$ be the product of Lebesgue measures in each factor.  Denote $S_n(t) = \sum_{k\neq n} t_k\alpha_k e_k$, $t\in H_n$,  and define
 $$
 \phi_n = \int_{H_n} \delta\left(\frac12e_n +S_n(t)\right) - \delta\left(-\frac12e_n+S_n(t)\right)\,d\lambda_n(t).
 $$
 This Bochner integral is well defined, since $S_n$ has compact range. 
 Let $E$ be the linear span of $\{e_n\}$, and let $R:E\to \F(X)$ be the linear operator defined by $R(e_n)=\phi_n$. Note that $\beta R = Id_E$ since for each $n$,
 
 $$
 \beta (\phi_n) = \int_{H_n} \beta\left[\delta\left(\frac12e_n +S_n(t)\right) - \delta\left(-\frac12e_n+S_n(t)\right)\right]\,d\lambda_n(t)=\int_{H_n} \frac12e_n  +  \frac12e_n\,d\lambda_n(t) = e_n. 
 $$
 Let $f$ be a Lipschitz Gate\^aux differentiable function on $X$. We have that  
$$ 
f\left(\frac12 e_n +S_n(t)\right) - f\left(-\frac12 e_n +S_n(t)\right)=\int_{-\frac12}^\frac12 \langle \nabla f(se_n +S_n(t),e_n\rangle\,ds,$$ 
thus applying Fubini's theorem we get that
\begin{align*}
\langle f,\phi_n\rangle & = \int_{H_n}\int_{-\frac12}^\frac12 \langle \nabla f(se_n +S_n(t),e_n\rangle\,ds\,d\lambda_n(t)\\
& =\int_H \langle \nabla f(L(t),e_n\rangle d\lambda(t). 
\end{align*}
Now since for each $y,x\in X$ we have $|\langle \nabla f(y),x \rangle|\leq \|f\|_{\mathrm{Lip}} \|x\|$, it follows that for each $x\in E$
$$
|\langle f,R(x)\rangle|  \leq \int_H |\langle \nabla f(L(t)), x\rangle| d\lambda(t)
\leq \|f\|_{\mathrm{Lip}} \|x\|.
$$
By \cite[Corollary 6.43]{benyamini1998geometric}, the Gate\^aux differentiable functions in $B_{\mathrm{Lip}_0(X)}$ are uniformly dense, thus $w^*$-dense. Then for each $x\in E$ we have that $\|R(x)\|\leq \|x\|$, and it follows that we can extend $R$ to a norm one operator $T:X\to\F(X)$ such that  $\beta T = Id_X$. 

It remains to show that $T$ is $U$-equivariant. Let $g\in U$, and consider the map $\psi_g: \N\to \{-1,1\}$ given by 
$$
\psi_g(n) = \begin{cases} 1, & ge_n = e_n\\
-1, & g e_n = -e_n.
\end{cases}
$$
Consider also, for each $n\in\N$, $\Psi_g^n: H_n\to H_n$ defined by $\Psi_g^n((t_j)_{j\neq n}) = (\psi_g(j)t_j)_{j\neq n}$. 
For each $n\in \N$, 
\begin{align*}
\tilde g  Te_n & 
 = \int_{H_n} \delta\left(g\left(\frac12e_n +S_n(t)\right)\right) - \delta\left(g\left(-\frac12e_n+S_n(t)\right)\right)\,d\lambda_n(t)\\
& = \int_{H_n} \delta\left(\frac12 \psi_g(n)e_n +S_n(\Psi_g^n (t))\right) - \delta\left(-\frac12 \psi_g(n)e_n+S_n(\Psi_g^n (t))\right)\,d\lambda_n(t)\\
 & = \int_{H_n} \delta\left(\frac12 \psi_g(n)e_n +S_n(t)\right) - \delta\left(-\frac12 \psi_g(n)e_n+S_n( t)\right)\,d\lambda_n(t)
 \\
 & = \int_{H_n} \psi_g(n)\left[\delta\left(\frac12 e_n + S_n(t)\right) - \delta\left(-\frac12 e_n+S_n( t)\right)\right]\,d\lambda_n(t)\\
 &= \psi_g(n) Te_n = Tge_n. 
\end{align*}
Then by linearity, $\tilde g Tx=Tgx$ for all $x\in E$. Since $E$ is dense in $X$, it follows that $\tilde g T = Tg$.

 \end{proof} 

\begin{question} Can the above proof be adapted to the more general case of a K\"othe space, with its group of units?  Or to the case of a space with a symmetric basis and the group $G$ acting by signs and permutations of the vectors of the basis?
\end{question}

\section{Beyond SOT-compact groups}

We turn to results related to the amenability or topological amenability of $G$.
If $G$ is a bounded group of automorphisms on a space $Y$, then $Y$ is said to be complemented in its bidual by a $G$-equivariant projection $\pi$ if $g\pi=\pi g^{**}$ for any $g \in G$. This notion appears in \cite{CF} where such a space $Y$ is also called a \emph{$G$-ultrasummand}.
Examples are: reflexive spaces with any bounded group $G$ of isomorphisms; more generally dual spaces $Y=X^*$ with groups of isomorphisms on $X^*$ which are adjoint to   isomorphisms on $X$, i.e. $\{ g^* , g \in G\} \subseteq \mathcal{L}(X^*)$; or $L_1(0,1)$  with its isometry group.

\subsection{Case of an amenable group}

In \cite{CF}  it is proved that an exact sequence of $G$-spaces, in which $G$ is amenable and the middle space is a $G$-ultrasummand, must $G$-split as soon as it splits as a sequence of Banach spaces (Theorem 10.7). Applying this result to the exact sequence and dual sequence relating $X$ and $\F(X)$, together with the fact that a dual space $Y^*$ is always a $G$-ultrasummand for the dual action on it of a group $G$ acting on $Y$, we obtain:

\begin{prop} Let $G$ be an amenable group and $X$ a $G$-space. Then:

\begin{itemize}
\item[(a)] $X$ has the dual $G$-equivariant lifting property.
    \item[(b)]
If $X$ has the lifting property and ${\mathcal F}(X)$ is a $\tilde{G}$-ultrasummand, then $X$ has the $G$-equivariant lifting property.
\end{itemize}
\end{prop}

This applies to all abelian groups $G$, so in particular to  the group $U$ of units of a K\"othe space (including when $U$ is not compact). In particular:

\begin{corollary} Let $1 \leq p \leq +\infty$. Then
\begin{enumerate}
\item
$L_p$ has the dual $U$-equivariant lifting property, i.e. there exists a projection map $P$ from ${\rm Lip}_0(L_p)$ onto $L_p^*$
such that $P(f)=f/\|f\| P(|f|)$ for all $f \in X$.
\item If ${\mathcal F}(L_p)$ is $U$-complemented in its bidual, then $L_p$ has the $U$-equivariant lifting property.
\end{enumerate}
\end{corollary}

\subsection{Increasing unions of compact groups}

We now consider the case when $G$ is the closure of an increasing union of compact sets $G_n$. An example of this are the unitary group ${\mathcal U}(\mathcal H)$, which may be written as the closure of the $n$-th dimensional unitary groups ${\mathcal U}(\ell_2^n)$ ; actually the unitary group has the stronger property of being a Levy group, see \cite{pestov} Chapter 4. For similar reasons, for $1 \leq p <+\infty$, the groups ${\rm Isom}(L_p[0,1])$ and
 ${\rm Isom}(\ell_p)$ are closure of increasing sequences of compact groups.

\begin{prop}\label{prop:dual lifting} Assume $G$ is a bounded group of automorphisms on $X$ that can be written as the SOT closure of an increasing sequence of SOT compact groups $G_n$. Assume furthermore that
the induced $\widetilde{G}^*$ on $Lip_0(X)$ is the SOT closure of the groups $\widetilde{G_n}^*$. 
If $X$ has the lifting property
then  $X$ has $G$-equivariant dual lifting property. \end{prop}
 
\begin{proof} We know there is 
a $G_n$-equivariant projection $P_n=T_n^*$ from
${\rm Lip}_0(X)$ onto $X^*$, associated to a $G_n$-equivariant lifting $T_n$.
We let $Pf=\lim_U^{w*} P_n f$, i.e.
$Pf(x)=\lim_U (P_nf)(x)$. This is again a bounded projection onto $X^*$, and it is $\cup_n G_n$ equivariant.
Finally if $g \in G$ then
$$P(\tilde{g}^* f)(x)=\lim_U (P_n \tilde{g}^* f) (x)= \lim_U  f(\tilde{g} T_n x).$$
By hypothesis $f  \tilde{g}$ may be approximated in the $Lip_0(X)$-norm by some $f  \tilde{g_k}$, in particular, since  $T_n x$ is uniformly bounded in $n$:
$$P(\tilde{g}^* f)(x) \simeq \lim_U f(\tilde{g_k} T_n x)= \lim_U f(T_n g_{k} x) \simeq \lim_U f(T_n gx)=\lim_U (P_n f)(gx)=(Pf)(gx),$$
so we have $G$-equivariance of $P$.
\end{proof}

Although ${\rm Isom}(L_p[0,1])$ and ${\rm Isom}(\ell_p)$ are the closure of increasing sequences of compact groups, it is unclear whether these groups satisfy the second condition in the above proposition. Thus the following question is still unanswered: 

\begin{question}
    Let $X$ be $L_p[0,1]$ or $\ell_p$, $1\leq p <\infty$, and let $G={\rm Isom}(X)$. Does $X$ satisfy the $G$-equivariant dual lifting property?
\end{question}

\begin{prop}\label{prop:projection implies lifting}  Assume $G$ is a bounded group of automorphisms on $X$, which can be written as the SOT closure of an increasing sequence of SOT compact groups $G_n$. Assume that $\F(X)$ is a $\tilde{G}$-ultrasummand. If $X$ has the lifting property then
 $X$  has the $G$-equivariant lifting property. \end{prop}

\begin{proof}  By Proposition \ref{SOTcompact} we know that there is a $G_n$-equivariant lifting $T_n: X \rightarrow \F(X)$ for each $n$.
 Let $T: X \rightarrow {\rm Lip}_0(X)^*$ be defined by
$Tx=\lim_U^{w^\ast} T_n x$ in ${\rm Lip}_0(X)^*$
(i.e. for any $f \in {\rm Lip}_0(X)$, $\langle Tx,f\rangle =\lim_U \langle T_n x,f\rangle $).

We claim that  $\pi T$, where $\pi:{\rm Lip}_0(X)^*\to \F(X)$ is a $G$-equivariant projection existing by assumption, is a linear lifting of $\beta$ (the $G$-equivariance of $\pi$ is not needed here). Indeed, for any $x \in X$, $\phi \in X^*$,
$$\langle \beta \pi Tx, \phi\rangle =\langle Tx, \pi^* \beta^* \phi\rangle $$
noting that $\beta^*$ is the natural inclusion of $X^*$ into ${\rm Lip}_0(X)$.
So,
$$\langle \beta \pi Tx, \phi\rangle =\lim_U \langle T_nx, \pi^*\beta^* \phi\rangle =\lim_U \langle T_nx, \beta^* \phi\rangle 
=\lim_U \langle \beta T_n x, \phi\rangle  =\langle x,\phi\rangle $$ proving the claim.

Note that if $g$ belongs to one of the $G_n$'s then
$$\langle Tgx,f\rangle =\lim_U \langle T_n gx, f\rangle =\lim_U \langle \tilde{g} T_n x, f\rangle 
=\lim_U \langle T_n x, f   g\rangle =\langle Tx, f   g\rangle =\langle \tilde{g}^{**} Tx, f\rangle. $$
Because $\pi$ is $\cup_n G_n$-equivariant, we see that
$\pi T$ is $\cup_n G_n$-equivariant.

To conclude, we prove that $\pi T$ is $G$-equivariant. 
So let $g \in G$ and fix $x \in X$. We have that
$$\pi Tgx  \simeq  \pi Tg_n x=\tilde{g_n} \pi Tx$$ for $n$ large enough. Now we use that $g \mapsto \tilde{g}$ is SOT-SOT continuous from
$G$ to ${\mathcal L}(\F(X))$, by Lemma \ref{SOTSOT}, to obtain that $\tilde{g_n} \pi Tx \simeq \tilde{g} \pi Tx$ for $n$ big enough. Hence,
$$\pi Tgx=\tilde{g} \pi Tx,$$
which concludes the proof.
\end{proof}

\begin{corollary} If $\F(\mathcal H)$ is a ${\mathcal U}(\mathcal H)$-ultrasummand, then the Hilbert space $\mathcal H$ has the 
${\mathcal U}(\mathcal H)$-equivariant lifting property, i.e. there exists a ${\mathcal U}(\mathcal H)$-equivariant right inverse of $\beta: \F(\mathcal H) \rightarrow {\mathcal H}$.
\end{corollary}

\begin{obs}
   We do not know much about Banach spaces $X$ for which $\F(X)$ is complemented in $\F(X)^*{}^*$. By \cite{cuth2019finitely}, for each $d$-dimensional Banach space $E$, $\F(E)$ is $d_{BM}(E,\ell_2^d)$-complemented in its bidual. In particular, $\F(\ell_2^d)$ is 1-complemented in its bidual. It is not known (according to that paper) whether $\F(\ell_2)$ is complemented in its bidual. As a consequence of Proposition \ref{prop:projection implies lifting}, if $X$ has the lifting property but we prove that there is no $G$-equivariant lifting from $X$ to $\F(X)$, then in particular there will be no $G$-equivariant projection from $\F(X)$** onto $\F(X)$.
\end{obs}

\subsection{Final questions}
\begin{question}
Suppose that $X$ has the $G$-equivariant lifting property, and that $\F(X)$ is an ultrasummand, i.e. is complemented in its bidual. Does it follow that $\F(X)$ is a $\tilde{G}$-ultrasummand, i.e. is complemented in its bidual by a $G$-equivariant projection? 
\end{question}

\begin{obs} 
The class of $G$-equivariant operators is a WOT-closed linear subspace of ${\mathcal L}(X,\F(X))$. The classes $\TT_C$ and $\TT_C(G)$ are convex  and closed in WOT.
\end{obs}

\begin{proof} 
Let $T_\alpha$ be a net of $G$-equivariant operators in ${\mathcal L}(X,\F(X))$ converging in WOT to $T\in {\mathcal L}(X,\F(X))$. Let $x\in X$, $f\in \rm{Lip}_0(X)$ and $g\in G$. Then, since $ f \tilde{g} \in \rm{Lip}_0(X)$, 
$$
\langle \tilde{g}Tx, f\rangle=\langle Tx, f \tilde{g}\rangle =\lim \langle T_{\alpha}x, f \tilde{g}\rangle=\lim \langle \tilde{g}T_{\alpha}x, f\rangle=\lim \langle T_{\alpha}gx, f\rangle = \langle Tgx, f\rangle. 
$$
Since $f$ is arbitrary, it follows that $\tilde{g} Tx=Tgx$. Thus $\tilde{g}T=Tg$. \end{proof}

 We have also the following dual results.

\begin{obs} The space ${\mathcal P}_C$ of (resp. $G$-equivariant)\ projections from
${\rm Lip}_0(X)$ onto $X^*$ of norm at most $C$ is ${\rm W}^*{\rm OT}$-compact.
\end{obs}
\begin{proof}

Indeed this is a subset of $(B_{X^*}(0,C), w*)^{B_{{\rm Lip}_0(X)}}$
via $P \mapsto (Pf)_f$
which is compact, and the notion of convergence induced on ${\mathcal P}$ is precisely the ${\rm W}^*{\rm OT}$-convergence. So we just need to check that ${\mathcal P}$ is closed.
So assume $P_\alpha$ converges to some $P$, linearity is clear, and
if $\phi \in X^* \subset {\rm Lip}_0(X)$, then $P(\phi)=\lim_\alpha P_\alpha(\phi)=\lim_\alpha \phi=\phi$, so $P \in {\mathcal P}$. The same holds when restricting to $G$-equivariant projections.
\end{proof}

It seems to remain open whether the above observations may be used together with separation theorems or general topological properties of $G$ to obtain existence of some $G$-equivariant lifting or projections.

\section*{Acknowledgements}
Parts of this research were carried out during the third author's visit to Universidade de S\~ao Paulo and Universidade Federal de S\~ao Paulo in 2022, for which she wishes to express her gratitude.

\bibliographystyle{siam}

\end{document}